\pdfoutput=1
\documentclass{IEEEtran}

\usepackage{textcomp}

\usepackage[T1]{fontenc}
\usepackage[utf8]{inputenc}
\usepackage{microtype}
\usepackage{amsmath}
\usepackage{amsthm}
\usepackage{amssymb}
\usepackage{mathtools}
\usepackage{thmtools}
\usepackage{commath}
\usepackage{bbm} 
\usepackage{verbatim} 
\usepackage{cite} 
\usepackage[capitalise]{cleveref}
\usepackage[font=small,labelfont=bf]{caption}
\usepackage[dvipsnames]{xcolor} 
\usepackage{enumitem}

\DeclareUnicodeCharacter{00A0}{ } 
\DeclarePairedDelimiterX\Set[2]{\lbrace}{\rbrace}%
 { #1 \,:\, #2 } 
\DeclarePairedDelimiterX\inprod[2]{\langle}{\rangle}%
 { #1 , #2 } 
\DeclarePairedDelimiterX\kld[2]{D_{\text{KL}}(}{)}%
 { #1 \: \| \: #2 } 

\newcommand{\maxeig}{\lambda_{\text{max}}}

\newcommand{\grad}{\nabla} 

\declaretheorem[Refname={Theorem,Theorems}]{theorem}
\declaretheorem[style=definition,numberlike=theorem,Refname={Definition,Definitions}]{definition}
\declaretheorem[numberlike=theorem,Refname={Lemma,Lemmas}]{lemma}

\declaretheorem[name=Proposition,numberlike=theorem,Refname={Proposition,Propositions}]{proposition}

\newcommand{\expec}{\mathbb{E}} 
\newcommand{\cov}{\operatorname{Cov}} 

\newcommand{\diag}{\operatorname{diag}} 
\newcommand{\gauss}{\mathcal{N}} 
\newcommand{\sigmap}[1][]{%
	\ifthenelse{\isempty{#1}}{\vctg{\mathcal{X}}}{\skew{4}\widehat{\vctg{\mathcal{X}}}}%
} 

\renewcommand{\epsilon}{\varepsilon}
\renewcommand{\phi}{\varphi}

\newcommand{\R}{\mathbb{R}} 

\newcommand{\vct}[1]{\mathbf{#1}} 
\newcommand{\vctg}{\pmb} 
\newcommand{\transpose}{\intercal} 
\newcommand{\trace}{\operatorname{tr}} 

\newcommand{\mA}{\vct{A}}
\newcommand{\mB}{\vct{B}}
\newcommand{\mC}{\vct{C}}

\newcommand{\mH}{\vct{H}}

\newcommand{\mW}{\vct{W}}

\newcommand{\mP}{\vct{P}}
\newcommand{\mId}{\vct{I}}
\newcommand{\mL}{\vct{L}}

\newcommand{\mJ}{\vct{J}}
\newcommand{\mV}{\vct{V}}
\newcommand{\mX}{\vct{X}}
\newcommand{\mY}{\vct{Y}}
\newcommand{\mU}{\vct{U}}
\newcommand{\mZ}{\vct{Z}}
\newcommand{\mZero}{\vct{0}}

\newcommand{\mb}{\vct{b}}
\newcommand{\mx}{\vct{x}}
\newcommand{\my}{\vct{y}}

\makeatletter
	\@ifundefined{md}{%
		}
		{
			
		}
\makeatother

\newcommand{\mm}{\vct{m}}
\newcommand{\mq}{\vct{q}}

\newcommand{\mf}{\vct{f}}
\newcommand{\mg}{\vct{g}}

\newcommand{\mh}{\vct{h}}

\newcommand{\mSigma}{\vctg\Sigma}

\newcommand{\muug}{\vctg\mu}

\allowdisplaybreaks

\usepackage{tikz}
\usepackage{pgfplots}
\usetikzlibrary{positioning}
\pgfplotsset{compat=newest}
\pgfplotsset{plot coordinates/math parser=false}
\newlength{\figwidth}
\newlength{\figheight}
\newcommand{\eqspace}{\hspace{0.5cm}}

\DeclareMathOperator{\lip}{Lip}

\newcommand{\was}{d_{\mathrm{W}}}
\DeclareMathOperator{\pr}{pr}

\newcommand{\etal}{\mbox{\emph{et al.\@}}}

\begin{document}
\title{Wasserstein bounds for non-linear Gaussian filters}
\author{Toni Karvonen and Simo Särkkä
  \thanks{T.\ Karvonen was supported by the Research Council of Finland grants 359183 and 368086.
  He also acknowledges the research environment provided by ELLIS Institute Finland.}  
\thanks{T.\ Karvonen is with the School of Engineering Sciences, Lappeenranta--Lahti University of Technology LUT, Lappeenranta, Finland (email: toni.karvonen@lut.fi).}
\thanks{S.\ Särkkä is with the Department of Electrical Engineering and Automation, Aalto University, Espoo, Finland(email: simo.sarkka@aalto.fi).}
}

\maketitle

\begin{abstract}
  Most Kalman filters for non-linear systems, such as the unscented Kalman filter, are based on Gaussian approximations.
  We use Poincaré inequalities to bound the Wasserstein distance between the true joint distribution of the prediction and measurement and its Gaussian approximation. 
  The bounds can be used to assess the performance of non-linear Gaussian filters and determine those filtering approximations that are most likely to induce error.
\end{abstract}

\begin{IEEEkeywords}
  Kalman filtering, Gaussian filtering, non-linear filtering, Wasserstein distance, Poincaré inequalities
\end{IEEEkeywords}

\section{Introduction}

\IEEEPARstart{T}{he} Kalman filter and its non-linear extensions have been ubiquitous tools for estimating the latent state of partially observed state-space Markov models since the 1960s. Their applications include, to name but a few relatively recent ones, various types of tracking such as vision-aided~\cite{LiKimMourikis2013} and pedestrian tracking~\cite{Foxlin2005}, elimination of noise from brain imaging data~\cite{SarkkaSolinNummenmaaVehtariAuranenVanniLin2012}, and simultaneous localization and mapping~\cite{DissanayakeNewmanClarkDurrantWhyteCsorba2002}.

The Kalman filter is the optimal linear estimator in the mean-square sense when the system is linear Gaussian~\cite{AndersonMoore1979}.
Moreover, under natural observability and controllability assumptions the linear Kalman filter exhibits strong exponential stability properties; see~\cite{Jazwinski1970, AndersonMoore1981}, and~\cite[App.\@~C]{KamenSu1999}.
However, useful theoretical results for popular extensions of the Kalman filter for non-linear Gaussian systems, such as the extended and unscented Kalman filters, have proved far more elusive.
While Lipschitzianity assumptions suffice for understanding the propagation of the filtering distribution through the system dynamics, the update step in which the distribution is conditioned on new measurements is much more difficult to analyse---even over one time-step.
Accordingly, most error analysis for non-linear Kalman filters requires essentially unverifiable assumptions on error covariances and other filtering quantities (e.g., \cite{ReifGuntherYazUnbehauen1999, XiongZhangChan2006, He2022}) or assumes that the measurement model is linear~\cite{Karvonen2020}.

Many non-linear Kalman filters are based on the idea of Gaussian moment-matching.
These \emph{Gaussian filters} replace all distributions that occur in filtering with moment-matched Gaussians that are far easier to operate with; we defer a detailed description to \Cref{sec:KF}.
To moment-match one has to compute integrals that are rarely available in closed form.
This has given rise to a class of non-linear filters that combine Gaussian moment-matching with numerical integration (see~\cite{WuHuWuHu2006} and~\cite[Ch.\@~8]{Sarkka2013}).
The unscented Kalman filter is undoubtedly the most popular member of this class of filters.

\subsection{Prior work and contributions}

In this article we study the accuracy of Gaussian filters over one time-step.
The only prior work on assessing the effect of moment-matching we are aware of is by Morelande and García-Fernández~\cite{MorelandeGarciaFernandez2013}.
By assuming a prediction distribution $\mX \sim \gauss(\mm,\mP)$ (i.e., $\mX$ is a Gaussian random vector with mean $\mm$ and covariance $\mP$) and a non-linear measurement model $\mY = \mh(\mX) + \mV$, where $\mV \sim \gauss(\mZero,\mSigma_\mV)$, they compute that the Kullback--Leibler divergence between the moment-matched Gaussian approximation $\widetilde{\mZ} = (\mX,\widetilde{\mY})$ and the true joint distribution $\mZ = (\mX,\mY)$ is\footnote{This is Eq.\ (23) of~\cite{MorelandeGarciaFernandez2013}, which erroneously omits the coefficient $1/2$.}
\begin{equation}\label{eq:KLbound}
\begin{split}
\kld{\mZ}{\widetilde{\mZ}} ={}& \frac{1}{2}\log\big[\det\big(\mId + \mSigma_\mV^{-1}\big[\cov[\mY] - \mSigma_\mV\big.\big.\Big.\\
&\quad- \Big.\big.\big. \cov(\mX,\mY)^\transpose\mP^{-1}\cov(\mX,\mY)\big]\big)\big].
\end{split}
\end{equation}
However, since the Kullback--Leibler divergence is not a metric, it is difficult to incorporate other errors, such as those propagated from previous time-steps or due to numerical integration in an unscented Kalman filter.

The aim of this article is to obtain upper bounds like~\eqref{eq:KLbound} but in terms of the Wasserstein distance (see \Cref{sec:wasserstein}) rather than the Kullback--Leibler divergence.
To bound the Wasserstein distance between a non-linearly transformed Gaussian distribution and its moment-matched approximation we use recent second-order Poincaré inequalities, to be reviewed in \Cref{sec:poincare}.
In \Cref{sec:WasBounds}, we derive an upper bound, given in \Cref{thm:mainTheorem}, for the Wasserstein distance between the true joint distribution of the prediction and measurement and that produced by the exact Gaussian Kalman filter.
Unlike in~\cite{MorelandeGarciaFernandez2013}, our bound is \emph{over both the prediction and update step}.
The bound is given in terms of expectations of the first and second derivatives of the dynamic and measurement model functions as well as the difference of real and approximated joint means and covariances.
The bound vanishes when the model is affine, in which case the prediction and measurement distributions are Gaussian.
\Cref{sec:numerical-example} contains a numerical example.
While the bound is very conservative, our results will help in gauging the worst-case performance of Gaussian filters in safety-critical applications and in determining which approximations in non-linear filtering are most likely to induce large errors.
We also hope that the results will bring about a new type of stability analysis for Gaussian filters.

\subsection{Notational conventions}

Gradient of a function $g \colon \R^d \to \R$ is a column-vector function $\grad g$ with elements $[\grad g]_i = \partial g/ \partial x_i$. For clarity, the gradient of a function $g$ evaluated at $\mf(\mx)$ is occasionally written $(\grad g)[\mf(\mx)]$.
The Hessian has elements $[\mH_g]_{ij} = \partial g / \partial x_i \partial x_j$. For a function $\mg\colon\R^d\to\R^d$ with components $g_i$, the Jacobian has elements $[\mJ_{\mg}]_{ij} = \partial g_i / \partial x_j$.

We use the Löwner partial ordering of positive-semidefinite matrices. For matrices $\mA$ and $\mB$, $\mA > \mB$ ($\mA \geq \mB$) means that $\mA - \mB$ is positive-definite (positive-semidefinite).
See~\cite[Ch.\@~8]{Bernstein2009} for a review of properties of the Löwner ordering.

Norm is denoted $\norm[0]{\cdot}$. Depending on the object, this stands for the Euclidean norm $\norm[0]{\mx}$ of a vector $\mx \in \R^d$ or the spectral norm of a matrix $\mA \in \R^{d\times d}$, which is defined as
\begin{equation*}
\norm[0]{\mA} = \sup\Set{\norm[0]{\mA\mx} / \norm[0]{\mx}}{\mZero \neq \mx \in \R^d} = \sqrt{\smash[b]{\maxeig(\mA^\transpose\!\mA)}},
\end{equation*}
where $\maxeig$ denotes the largest eigenvalue.
The $L^p$-norm $\norm[0]{\mX}_p$ of a random variable $\mX$ is \smash{$\norm[0]{\mX}_p = (\expec[\norm[0]{\mX}^p])^{1/p}$}.
If $\mA$ is positive-definite, $\sqrt{\mA}$ stands for the unique positive-definite matrix such that $\sqrt{\mA} \sqrt{\mA} = \mA$.

\section{Non-linear Gaussian filtering} \label{sec:KF}

In this article we consider a non-linear dynamic system
\begin{align}
\mX &= \mf(\mX_p) + \mU, \label{eq:X} \\ 
\mY &= \mh(\mX) + \mV , \label{eq:Y}
\end{align}
where the \emph{state} $\mX \in \R^n$ is obtained by propagating the $n$-dimensional Gaussian prior $\mX_p \sim \gauss(\mm_p, \mP_p)$ through the \emph{dynamic model function} $\mf\colon\R^n\to\R^n$.
The \emph{measurement} \sloppy{${\my \in \R^m}$} is a realization of the random variable $\mY$ that is obtained via the \emph{measurement model function} $\mh\colon\R^n\to\R^m$.
The state and measurement are corrupted by additive Gaussian noises $\mU \sim \gauss(\mZero,\mSigma_\mU)$ and $\mV \sim \gauss(\mZero,\mSigma_\mV)$.
The covariance matrices $\mP_p$, $\mSigma_\mU$, and $\mSigma_\mV$ are assumed positive-definite and the random variables $\mX_p$, $\mU$ and $\mV$ independent. The distribution of $\mX$ is called the \emph{prediction distribution}. Unfortunately, the \emph{filtering distribution} $\mX \mid \mY = \my$ is Gaussian only if the dynamic and measurement model functions are affine.

We use $\pr$ to denote Gaussian projection that transforms a given random variable into a Gaussian one while preserving the mean and covariance.
That is, for an arbitrary random variable $\mZ$ with finite mean $\expec(\mZ)$ and covariance $\cov(\mZ)$ we have
\begin{equation} \label{eq:gauss-pr}
  \pr \mZ \sim \gauss( \expec(\mZ), \cov(\mZ) ) ,
\end{equation}
which is to say that the moments of the Gaussian projection match those of the original random variable.
In Gaussian filtering~\cite[Ch.\@~8]{Sarkka2013}, a Gaussian approximation to the filtering distribution is obtained via two projections:
\begin{enumerate}
\item[1.] \emph{Prediction.} Approximate $\mX$ in~\eqref{eq:X} with the moment-matched Gaussian $\widetilde{\mX} = \pr \mX$. That is,
\begin{equation*}
p(\mX) \approx p(\widetilde{\mX}) = \gauss(\mm_\mX,\mP_\mX) ,
\end{equation*}
where $\mm_\mX = \expec(\mX) = \expec[\mf(\mX_p)]$ and
\begin{equation*}
\mP_\mX = \cov(\mX) = \cov[\mf(\mX_p)] + \mSigma_\mU.
\end{equation*}
\item[2.] Use $\widetilde{\mX}$ to construct a Gaussian approximation
  \begin{equation*}
    \widetilde{\mZ} = (\widetilde{\mX}, \widetilde{\mY}) = \pr (\widetilde{\mX}, \mh(\widetilde{\mX}) + \mV)
  \end{equation*}
  of the joint distribution $\mZ$ of $\mX$ and $\mY$.
  The Gaussian approximation to $\mY$ is then \smash{$\widetilde{\mY} \sim \gauss(\mm_{\widetilde{\mY}},\mP_{\widetilde{\mY}})$}, where \smash{$\mm_{\widetilde{\mY}} = \expec(\widetilde{\mY}) = \expec[\mh(\widetilde{\mX})]$} and
\begin{equation*}
\mP_{\widetilde{\mY}} = \cov(\widetilde{\mY}) = \cov[\mh(\widetilde{\mX})] + \mSigma_\mV.
\end{equation*}
Note that $\widetilde{\mY}$ does \emph{not} equal $\pr \mY$:
\begin{equation*}
  \widetilde{\mY} = \pr ( \mh(\pr \mX) + \mV) \neq \pr(\mh(\mX) + \mV) = \pr \mY.
\end{equation*}
The cross-covariance is
\begin{equation*}
  \mC_{\widetilde{\mX},\widetilde{\mY}} = \cov(\widetilde{\mX},\widetilde{\mY}) = \expec\big[ (\widetilde{\mX} - \mm_\mX)(\mh(\widetilde{\mX}) - \mm_{\widetilde{\mY}})^\transpose \big]. 
\end{equation*}
\item[3.] \emph{Update.} Use the standard Gaussian conditioning formulas to compute a Gaussian approximation to $\mX \mid \mY = \my$:
\begin{equation*} 
\begin{split} 
p(\mX \mid \mY = \my) &\approx p(\widetilde{\mX} \mid \widetilde{\mY} = \my) = \gauss\big(\mm_{\widetilde{\mX} \mid \widetilde{\mY}}, \mP_{\widetilde{\mX} \mid \widetilde{\mY}}\big) ,
\end{split}
\end{equation*}
where
\begin{align*}
\mm_{\widetilde{\mX} \mid \widetilde{\mY}} &= \mm_\mX + \mC_{\widetilde{\mX},\widetilde{\mY}}\mP_{\widetilde{\mY}}^{-1}\big(\mY -  \mm_{\widetilde{\mY}}\big),\\
\mP_{\widetilde{\mX} \mid \widetilde{\mY}} &= \mP_\mX - \mC_{\widetilde{\mX},\widetilde{\mY}} \mP_{\widetilde{\mY}}^{-1} \mC_{\widetilde{\mX},\widetilde{\mY}}^\transpose.
\end{align*}
\end{enumerate}

\Cref{fig:kf-structure} shows the one-step structure of a Gaussian filter.
The Gaussian approximation to the conditional distribution is then treated as $\mX_p$ and the procedure repeated for as many time-steps as called for.
Within this scheme the Gaussian integrals
\begin{equation*}
\int_{\R^n}\mg_1(\mx) \gauss(\mx \mid \muug,\mSigma), \eqspace \int_{\R^n}\mg_1(\mx)\mg_2(\mx)^\transpose \gauss(\mx \mid \muug,\mSigma),
\end{equation*}
for $\mg_1, \mg_2 \in \{\mf,\mh\}$ need to be computed, a task that is rarely possible analytically. In engineering literature, the most prevalent numerical integration methods for approximating these integrals are the fully symmetric cubature integration formulas of McNamee and Stenger~\cite{McNameeStenger1967}, used in the unscented transform of Julier \etal~\cite{JulierUhlmannDurrant1995}. See~\cite{WuHuWuHu2006} and~\cite[Ch.\@~8]{Sarkka2013} for other popular sigma-point methods. As our aim is not to compare accuracy or effect of different sigma-point schemes, the problem of integral computations is not taken into account in what follows (one interested in this aspect should see, e.g.,\@~\cite{MorelandeGarciaFernandez2013} and~\cite{FarinaRisticBenvenuti2002}).
That is, we consider an \emph{exact Gaussian filter} which computes the moments exactly.
\cref{thm:2ndPoincare}, our main tool, would not even lend itself easily to inclusion of numerical integration errors.
The triangle inequality could be used, but this would offer little additional insight.

\section{Wasserstein distance} \label{sec:wasserstein}

This section introduces the main metric of this article, the Wasserstein distance. 
Aside from the metrics discussed here, a great number of other metrics for measuring the distance between random variables exist~\cite{GibbsSu2002}.
We focus on the Wasserstein distance because Theorem~\ref{thm:2ndPoincare}, our main workhorse, is not available in sufficient generality for other distances.
The $p$th \emph{Wasserstein distance} $\text{W}_p$, often known as earth mover's distance, between random variables $\mX$ and $\mZ$ in $\R^d$ is
\begin{equation*}
\text{W}_p(\mX,\mZ)^p = \inf\expec(\norm[0]{\mY - \mU}^p) ,
\end{equation*}
where the infimum is over all joint distributions of random variables $\mY$ and $\mU$ whose marginal distributions coincide with those of $\mX$ and~$\mZ$.
Wasserstein distances are important in the theory of optimal transport~\cite{Villani2008}.
We use only the first Wasserstein distance $\text{W}_1$ that will be from now simply called the Wasserstein distance and denoted $\was$. 

Many useful probability metrics are encompassed by the family of metrics defined as
\begin{equation}\label{eq:probabilityMetric}
d(\mX,\mZ) = \sup_{h \in \mathcal{H}} \abs[1]{\expec[h(\mX)] - \expec[h(\mZ)]} ,
\end{equation}
where $\mathcal{H}$ is some class of functions from $\R^d$ to $\R$.
The two most important members of this class are the Wasserstein distance
\begin{equation*}
  \was(\mX,\mZ) = \text{W}_1(\mX, \mZ) = \sup_{\lip(h) \leq 1} \abs[1]{\expec[h(\mX)] - \expec[h(\mZ)]} ,
\end{equation*}
where $\lip(h) = \sup_{\mx\neq\my} \abs[0]{h(\mx) - h(\my)} / \norm[0]{\mx - \my}$ is the \emph{maximal Lipschitz constant}, and the \emph{total variation distance}
\begin{equation*}
d_{\text{TV}}(\mX,\mZ) = \sup_{\abs[0]{h} \leq 1/2} \abs[1]{\expec[h(\mX)] - \expec[h(\mZ)]} .
\end{equation*}
Note that the total variation distance is at most one and is thus more suitable for convergence analysis than for our purposes.

The following propositions can be found in~\cite{GivensShortt1984}.
When combined, they provide an upper bound on the Wasserstein distance between two Gaussian random variables.

\begin{proposition} \label{thm:WasLpInEq}
  If $1 \leq q \leq p$, then $\mathrm{W}_q(\mX,\mZ) \leq \mathrm{W}_p(\mX,\mZ)$ for any random variables $\mX$ and~$\mZ$.
\end{proposition}

\begin{proposition} \label{thm:W2Gaussians}
  Let $\mX \sim \gauss(\muug_\mX,\mSigma_\mX)$ and \sloppy{${\mZ \sim \gauss(\muug_\mZ,\mSigma_\mZ)}$} with positive-definite $\mSigma_\mX$ and $\mSigma_\mZ$.
  Set $\mL_\mX = \mSigma_\mX^{1/2}$.
  Then
\begin{equation*} 
\begin{split}
\mathrm{W}_2(\mX,\mZ)^2 ={}& \norm[0]{\muug_\mX - \muug_\mZ}^2 + \trace\mSigma_\mX + \trace\mSigma_\mZ \\
&- 2\trace\big(\sqrt{\smash[b]{\mL_\mX\mSigma_\mZ\mL_\mX}} \, \big).
\end{split}
\end{equation*}
\end{proposition}

In the notation of \cref{thm:W2Gaussians}, these propositions yield
\begin{equation}\label{eq:WasGBound}
\begin{split}
  \was(\mX,\mZ)^2 \leq{}& \norm[0]{\muug_\mX - \muug_\mZ}^2 + \trace\mSigma_\mX + \trace\mSigma_\mZ \\
& - 2\trace\big(\sqrt{\smash[b]{\mL_\mX\mSigma_\mZ\mL_\mX}} \, \big).
\end{split}
\end{equation}
Ley \etal~\cite{LeyReinertSwan2015} have derived an upper bound for the Wasserstein distance between two univariate Gaussians directly, without using \cref{thm:WasLpInEq}. However, in the multivariate case bounds derived in this way seem to exist only for centered Gaussians~\cite[Exr.\@~4.5.3]{NourdinPeccati2012}.

\section{Stein's method and Poincaré inequalities} \label{sec:poincare}

This section introduces the two inequalities essential for the developments of this article. These are the \emph{Poincaré inequalities} of \cref{thm:1stPoincare,thm:2ndPoincare} that provide means of controlling a transformed Gaussian random variable in terms of expectations of first and second derivatives of the transformation. These inequalities, particularly the second one, are closely linked to \emph{Stein's method}, the basics of which we sketch next.

Stein's method has its basis on the familiar identity
\begin{equation} \label{eq:stein-identity}
\expec[X g(X)] = \expec[g'(X)],
\end{equation}
often called \emph{Stein's identity} or \emph{Stein's lemma}, that holds for every differentiable function \sloppy{${g\colon\R\to\R}$} if and only if \sloppy{${X \sim \gauss(0,1)}$}.
This identity allows one to assess the Gaussianity of a transformed random variable via the differential equation
\begin{equation*}
g'(z) - zg(z) = h(z) - \expec[h(X)],
\end{equation*}
where $h\colon\R\to\R$ is differentiable.
For a solution $g_h$ to this differential equation and any random variable $Z$ we have
\begin{equation*}
\expec[h(X)] - \expec[h(Z)] = \expec[g_h'(Z)] - \expec[Zg_h(Z)] ,
\end{equation*}
where the right-hand side vanishes by~\eqref{eq:stein-identity} if $Z \sim \gauss(0, 1)$.
In the univariate case, \Cref{eq:probabilityMetric} can be thus written as
\begin{equation*}
\begin{split}
d_{\mathcal{H}}(X,Z) &= \sup_{h \in \mathcal{H}} \abs[1]{\expec[h(X)] - \expec[h(Z)]} \\
&= \sup_{h \in \mathcal{H}} \abs[1]{\expec[g_h'(Z)] - \expec[Z g_h(Z)]},
\end{split}
\end{equation*}
which depends on $X$ only via $g_h$.
Stein's method has its origins in the work of Stein~\cite{Stein1972}. See~\cite{ChenGoldsteinShao2011, NourdinPeccati2012}, and \cite{LeyReinertSwan2017} for reviews.

First of the inequalities we are interested in is variously known as \emph{Poincaré inequality}~\cite{NourdinPeccati2012} or \emph{Chernoff's inequality}~\cite{Cacoullos1982} and provides variance bounds for transformed Gaussian random variables. This inequality originates in the works of Nash~\cite[Part II]{Nash1958} and Chernoff~\cite{Chernoff1981}, with proofs based on a variance expansion in terms of Hermite polynomials. 
The multivariate version below is due to Cacoullos~\cite{Cacoullos1982}.

\begin{theorem}[Poincaré inequality; \cite{Cacoullos1982}]\label{thm:1stPoincare} 
If $\mg \colon \R^{d_1} \to \R^{d_2}$ is differentiable and $\mX \sim \gauss(\muug,\mSigma)$, then
\begin{equation*}
\expec[\mJ_{\mg}(\mX)]  \mSigma \expec[\mJ_{\mg}(\mX)]^\transpose \leq \cov[\mg(\mX)] \leq \expec\big[\mJ_{\mg}(\mX)\mSigma\mJ_{\mg}(\mX)^\transpose\big].
\end{equation*}
The bounds coincide if and only if $\mg$ is affine, which is to say that $\mg(\mx) = \mA\mx + \mb$ for some $\mA \in \R^{d_2\times d_1}$ and $\mb \in \R^{d_2}$.
\end{theorem}

Recent \emph{second order Poincaré inequalities} are essential for our purposes. 
These inequalities bound distances between $\mg(\mX)$ and $\pr \mg(\mX)$ when $\mX$ is Gaussian.
The first such inequality, proved by Chatterjee~\cite{Chatterjee2009}, was for the total variation distance but required that $\mg$ be a mapping to the real line.
Multivariate extensions for the Wasserstein distance appear in~\cite{NourdinPeccatiReinert2009} and~\cite{NourdinPeccatiReveillac2010}. The version in \Cref{thm:2ndPoincare} below is a corollary of~\cite[Thm.\@~7.1]{NourdinPeccatiReinert2009}.

\begin{definition}\label{def:class} 
Let $\mX \in \R^{d_1}$ be a random variable. The set of twice-differentiable functions $\mg\colon\R^{d_1}\to\R^{d_2}$ such that $\mg(\mX)$, $\grad g_i(\mX)$ and $\mH_{g_i}(\mX)$ have finite $L^4$-norms for all $i=1,\ldots,d_2$ is denoted by $\mathcal{L}(\mX)$.
\end{definition}

For notational convenience we define
  \begin{align*}
    \lVert \nabla \mg(\mX) \rVert_{4'} &= \sum_{i=1}^{d_2} ( \expec [\lVert \nabla g_i(\mX) \rVert^4] )^{1/4}, \\
    \lVert \mH_{\mg}(\mX) \rVert_{4'} &= \sum_{i=1}^{d_2} ( \expec [\lVert \mH_{g_i}(\mX) \rVert^4] )^{1/4}
  \end{align*}
when $\mg \colon \R^{d_1} \to \R^{d_2}$.
Note that due to the placement of the root these are not conventional $L^4$-norms but sums of $L^4$-norms.
These quantities are finite if $\mg \in \mathcal{L}(\mX)$.

\begin{theorem}[2nd order Poincaré inequality; \cite{NourdinPeccatiReinert2009}]\label{thm:2ndPoincare} Let $\mX \in \R^{d_1}$ be a Gaussian random variable and $\mg \colon \R^{d_1} \to \R^{d_2}$ for $d_1 \geq d_2$ a twice differentiable function. 
Suppose that $\mg\in\mathcal{L}(\mX)$ and that $\mg(\mX)$ has mean $\muug$ and positive-definite covariance $\mSigma$. 
Then
\begin{equation*} \label{ineq:2ndPoincare}
\begin{split}
\was\big(\mg(\mX),\pr \mg(\mX)\big) \leq \frac{3}{\sqrt{2}} & \norm[0]{\mSigma^{-1}} \norm[0]{\mSigma}^{1/2} \\
&\times\norm[0]{\mH_{\mg}(\mX)}_{4'} \norm[0]{\grad \mg(\mX)}_{4'} .
\end{split}
\end{equation*}
\end{theorem}

Interpreting this inequality is straighforward: if the Hessians are non-zero, the distance is positive since $\mg(\mX)$ is not Gaussian; if the Hessians vanish, $\mg$ is affine, so that $\mg(\mX)$ is Gaussian and $\pr \mg(\mX) = \mg(\mX)$.

\section{Wasserstein bounds for Gaussian filters}\label{sec:WasBounds}

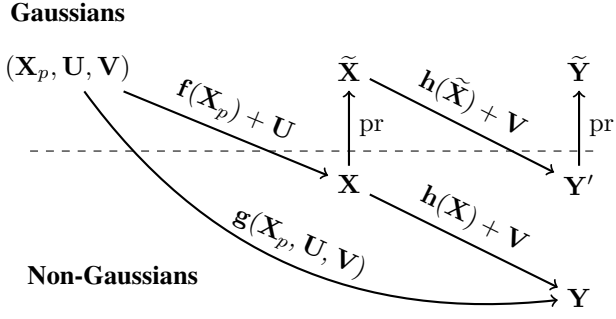
\begin{figure}
\centering
\begin{tikzpicture}
\node (Xp) [] {$(\mX_p,\mU,\mV)$};
\node (Xtilde) [right=of Xp,xshift=1.5cm] {$\widetilde{\mX}$};
\node (X) [below=of Xtilde] {$\mX$};
\node (Ytilde) [right=of Xtilde,xshift=1.5cm] {$\widetilde{\mY}$};
\node (Ydot) [below=of Ytilde] {$\mY'$};
\node (Y) [below=of Ydot] {$\mY$};

\draw[->,thick] (Xp) --(X) node[above,midway,sloped] {$\mf(\mX_p) + \mU$};
\draw[->,thick] (X) --(Xtilde) node[midway,xshift=0.3cm] {$\pr$};
\draw[->,thick] (Xtilde) --(Ydot) node[above,midway,sloped] {$\mh(\widetilde{\mX}) + \mV$};
\draw[->,thick] (Ydot) --(Ytilde) node[midway,xshift=0.3cm] {$\pr$};
\draw[->,thick] (X) --(Y) node[above,midway,sloped] {$\mh(\mX) + \mV$};

\draw[dashed] (-0.5,-1.1) -- (7,-1.1);

\path[->,thick] (Xp) edge[bend right, black!40] node [right,above,midway, yshift=0.15cm, xshift=0.5cm] {$\mg(\mX_p,\mU,\mV)$} (Y);
\path[->,thick] (Xp) edge[bend right = 32, black!40] (X);

\node (gauss) [above=of Xp,yshift=-0.8cm] {\textbf{Gaussians}};
\node (nongauss) [below=of Xp,yshift=-1.2cm,xshift=0.6cm] {\textbf{Non-Gaussians}};

\end{tikzpicture}
\caption{Structure and relations of different random variables and their Gaussian approximations in \cref{sec:KF,sec:WasBounds}. The random variables on the upper part of the figure are Gaussian while the ones on the lower part are not unless $\mf$ and $\mh$ are affine. The Gaussian projection $\pr$ is defined in~\eqref{eq:gauss-pr} and the function $\mg$ in \Cref{eq:geq}.}\label{fig:kf-structure}
\end{figure}

This section contains the main results of the article.
We establish upper bounds on the Wasserstein distance $\was(\mZ,\widetilde{\mZ})$ between the true joint distribution $\mZ$ of $\mX$ and $\mY$ and the Gaussian filter approximation $\widetilde{\mZ} = (\widetilde{\mX},\widetilde{\mY})$ constructed as described in \cref{sec:KF}.
To use \cref{thm:2ndPoincare} we need to write $\mZ$ as a transformation of a Gaussian random variable. This can be done by defining a transformation $\mg\colon\R^{2n + m}\to \R^{n + m}$ as
\begin{equation}\label{eq:geq}
\mg(\mx) = \mg\begin{pmatrix}\mx_1 \\ \mx_2 \\ \mx_3 \end{pmatrix} = \begin{pmatrix} \mf(\mx_1) + \mx_2 \\ \mq(\mx_1,\mx_2) + \mx_3 \end{pmatrix}
\end{equation}
where $\mx_1,\mx_2\in\R^n$, $\mx_3 \in \R^m$ and 
\begin{equation}\label{eq:furtherAuxFunc}
\mq( \mx_1 , \mx_2 ) = \mh[\mf(\mx_1) + \mx_2].
\end{equation}
It then follows that
\begin{equation*}
\mZ = ( \mX  , \mY ) = \mg(\mW) = \mg( \mX_p, \mU, \mV ),
\end{equation*}
where $\mW = (\mX_p, \mU, \mV) \in \R^{2n + m}$ is Gaussian.

As is apparent from the structure of a Gaussian filter depicted in \Cref{fig:kf-structure}, the mean and covariance of $\mZ$ do not match those of $\widetilde{\mZ}$, which means that \cref{thm:2ndPoincare} cannot be applied directly.\footnote{We note that one can construct, in a manner reminiscent of the \emph{augmented} unscented transform~\cite{WuHuWuHu2005}, a filter that is based on a direct Gaussian approximation of $\mZ = \mg(\mX_p, \mU, \mV)$ instead of an approximation derived from $\pr \mX$ (see Figure~\ref{fig:kf-structure}).
This is rarely done in practice.
However, in the augmented case it suffices to bound $\was(\mZ,\pr\mZ)$, which is done in Lemma~\ref{lemma:PoincareLemma}.} The triangle inequality and~\eqref{eq:WasGBound} tackle this problem:
\begin{equation}\label{ineq:triangle}
\was(\mZ,\widetilde{\mZ}) \leq \was(\mZ,\pr\mZ) + \was(\pr\mZ,\widetilde{\mZ}) .
\end{equation}
\Cref{thm:2ndPoincare} is now applicable to the first term and the bound~\eqref{eq:WasGBound} to the second.
We begin by bounding the true joint covariance $\cov(\mZ)$ with the Poincaré variance inequality of \Cref{thm:1stPoincare}.

\begin{lemma}\label{lemma:CovLemma} If $\mf$ and $\mh$ are differentiable, then the joint covariance $\cov(\mZ)$ obeys the Löwner bounds
\begin{equation}\label{eq:1stPoincareCovZ}
\mA_1 + \mB_1 + \mC_\mV \leq \cov(\mZ) \leq \mA_2 + \mC_\mV ,
\end{equation}
where
\begin{align}
\mA_1 ={}& \begin{pmatrix} \expec[\mJ_\mf(\mX_p)] \\ \expec\big[\mJ_\mh[\mf(\mX_p) + \mU]\mJ_\mf(\mX_p)\big] \end{pmatrix} \mP_p \nonumber \\
&\times \begin{pmatrix} \expec[\mJ_\mf(\mX_p)]^\transpose & \expec\big[\mJ_\mh[\mf(\mX_p) + \mU]\mJ_\mf(\mX_p)\big]^\transpose\end{pmatrix}, \nonumber \\
\mB_1 ={}& \begin{pmatrix} \mId_{n\times n} \\ \expec\big(\mJ_\mh[\mf(\mX_p) + \mU]\big) \end{pmatrix} \mSigma_\mU \nonumber \\
&\times \begin{pmatrix} \mId_{n\times n} & \expec\big(\mJ_\mh[\mf(\mX_p) + \mU]\big)^\transpose\end{pmatrix}, \nonumber \\
\mA_2 ={}& \expec\left[\begin{pmatrix} \mId_{n\times n} \\ \mJ_\mh[\mf(\mX_p) + \mU] \end{pmatrix} \big[\mJ_\mf(\mX_p)\mP_p\mJ_\mf(\mX_p)^\transpose + \mSigma_\mU\big]\right. \nonumber \\
&\quad\times \left.\begin{pmatrix} \mId_{n\times n} & \mJ_\mh[\mf(\mX_p) + \mU]^\transpose\end{pmatrix}\right], \nonumber \\
\mC_\mV ={}& \begin{pmatrix} \mZero_{n\times n} & \mZero_{n\times m} \\ \mZero_{m \times n} & \mSigma_\mV \end{pmatrix}. \label{eq:mC-def}
\end{align}
\end{lemma}
\begin{proof} The Jacobian of $\mg$ in \Cref{eq:geq} is
\begin{equation*} 
\mJ_\mg(\mx) = \begin{pmatrix} \mJ_\mf(\mx_1) & \mId_{n\times n} & \mZero_{n\times m} \\ \mJ_\mh[\mf(\mx_1) + \mx_2]\mJ_\mf(\mx_1) & \mJ_\mh[\mf(\mx_1) + \mx_2] & \mId_{m\times m} \end{pmatrix}
\end{equation*}
by the chain rule.
We can now use \Cref{thm:1stPoincare} with $\mX = \mW$. Because the prior and noise terms are assumed independent, the covariance matrix $\cov(\mW) = \diag(\mP_p,\mSigma_\mU,\mSigma_\mV)$ is block-diagonal. Using the expression for $\mJ_\mg$ above, we compute the block matrices $\expec[\mJ_{\mg}(\mW)] \cov(\mW) \, \expec[\mJ_{\mg}(\mW)]^\transpose$ and $\expec[\mJ_{\mg}(\mW)\cov(\mW) \, \mJ_{\mg}(\mW)^\transpose]$ needed in \Cref{thm:1stPoincare} and, after some factorization, arrive at~\eqref{eq:1stPoincareCovZ}.
\end{proof}

The following lemma takes care of $\was(\mZ,\pr\mZ)$.

\begin{lemma}\label{lemma:PoincareLemma} Suppose that $\mg$ in \Cref{eq:geq} is in $\mathcal{L}(\mW)$ from \cref{def:class}. Let $\mP_\mZ$ denote the covariance of $\mZ$. Then
\begin{equation*}
\begin{split}
\was({}&\mZ,\pr\mZ) \\
\leq{}& \frac{3}{\sqrt{2}} \norm[0]{\mP_\mZ^{-1}} \norm[0]{\mP_\mZ}^{1/2} \big(\norm[0]{\mH_{\mf}(\mX_p)}_{4'} + \norm[0]{\mH_{\mq}(\mX_p,\mU)}_{4'} \big) \\
&\times \Big(\sum_{i=1}^n\norm[1]{1 + \norm[0]{\grad f_i(\mX_p)}}_2^{1/2} \Big.\\
&\quad\quad+ \Big.\sum_{i=1}^m\norm[2]{1 + \norm[1]{\mJ_\mf^{\mId}(\mX_p)(\grad h_i)[\mf(\mX_p) + \mU]}}_2^{1/2}\Big) ,
\end{split}
\end{equation*}
where the matrix $\mJ_\mf^{\mId}(\mX_p)$ is defined in \Cref{eq:gradqi}.
\end{lemma}
\begin{proof} Because the mean and covariance of $\pr\mZ$ are equal to those of $\mZ$, we can apply \cref{thm:2ndPoincare} to $\mW$ and $\mg$. This yields
\begin{equation*}
\begin{split}
\was\big(\mZ,\mZ_{\mathrm{G}}\big) \leq{}& \frac{3\sqrt{2}}{2} \norm[0]{\mP_\mZ^{-1}} \norm[0]{\mP_\mZ}^{1/2} \sum_{i=1}^{n+m} \big(\expec[\norm[0]{\mH_{g_i}(\mW)}^4]\big)^{1/4} \\
&\times \sum_{i=1}^{n+m} \big(\expec[\norm[0]{\grad g_i(\mW)}^4]\big)^{1/4}.
\end{split}
\end{equation*}
The gradients and Hessians can be expanded. For $i \leq n$, $\mH_{g_i}$ is zero except for $\mH_{f_i}$ on its upper left corner so that $\norm[0]{\mH_{g_i}(\mW)} = \norm[0]{\mH_{f_i}(\mX_p)}$. For $i > n$, the Hessian has on its upper left corner the $2n\times 2n$ Hessian $\mH_{q_i}$ of $\mq$ in~\eqref{eq:furtherAuxFunc} and hence $\norm[0]{\mH_{g_i}(\mW)} = \norm[0]{\mH_{q_i}(\mX_p,\mU)}$.
For the first $n$ gradients,
\begin{equation*}
\norm[0]{\grad g_i(\mW)}^2 = 1 + \norm[0]{\grad f_i(\mX_p)}^2
\end{equation*}
and for the rest
\begin{equation*}
\norm[0]{\grad g_i(\mW)}^2 = 1 + \norm[0]{\grad q_i(\mX_p,\mU)}^2.
\end{equation*}
By the chain rule, the second gradient is
\begin{equation} \label{eq:gradqi-1}
\begin{split}
\grad q_i(\mX_p,\mU) &= \begin{pmatrix} \mJ_\mf(\mX_p)^\transpose (\grad h_i)[\mf(\mX_p) + \mU] \\ (\grad h_i)[\mf(\mX_p) + \mU]\end{pmatrix} \\
&= \mJ_\mf^{\mId}(\mX_p) (\grad h_i)[\mf(\mX_p) + \mU],
\end{split}
\end{equation}
where
\begin{equation} \label{eq:gradqi}
  \mJ_\mf^{\mId}(\mX_p) = \begin{pmatrix} \mJ_\mf(\mX_p)^\transpose \\ \mId_{n\times n}\end{pmatrix} \in \R^{2n \times n}.
\end{equation}
Recall the $L^p$-norm notation \smash{$\norm[0]{\mX}_p = (\expec[\norm[0]{\mX}^p])^{1/p}$}.
Therefore $(\expec[X^4])^{1/4} = \lVert X^2 \rVert_2^{1/2}$ for any univariate random variable $X$.
With this in mind, the bounds and identities above yield
\begin{equation*}
\begin{split}
\was({}&\mZ,\pr\mZ) \\
\leq{}& \frac{3}{\sqrt{2}} \norm[0]{\mP_\mZ^{-1}} \norm[0]{\mP_\mZ}^{1/2}  \big(\norm[0]{\mH_{\mf}(\mX_p)}_{4'} + \norm[0]{\mH_{\mq}(\mX_p,\mU)}_{4'} \big) \\
&\times \Big(\sum_{i=1}^n\norm[1]{1 + \norm[0]{\grad f_i(\mX_p)}^2}_2^{1/2} \Big.\\
&\quad\quad+ \Big.\sum_{i=1}^m\norm[1]{1 + \norm[0]{\grad q_i(\mX_p,\mU)}^2}_2^{1/2}\Big),
\end{split}
\end{equation*}
so that inserting \eqref{eq:gradqi-1} yields the claimed upper bound.
\end{proof}

We can now combine these lemmas with~\eqref{ineq:triangle} and~\eqref{eq:WasGBound} to obtain the main result of this article.

\begin{theorem}\label{thm:mainTheorem} If $\mg$ in \Cref{eq:geq} is in $\mathcal{L}(\mW)$, then
\begin{equation*}
  \was(\mZ,\widetilde{\mZ}) \leq C_1 + C_2,
\end{equation*}
where the two constants $C_1$ and $C_2$, which give the upper bounds $\was(\mZ,\pr\mZ) \leq C_1$ and $\was(\pr\mZ,\widetilde{\mZ}) \leq C_2$, are
\begin{equation*}
  \begin{split}
  C_1 ={}& \frac{3}{\sqrt{2}} \norm[1]{\big(\mA_1 + \mB_1 + \mC_\mV\big)^{-1}} \norm[0]{\mA_2 + \mC_\mV}^{1/2} \\
 &\times \big(\norm[0]{\mH_{\mf}(\mX_p)}_{4'} + \norm[0]{\mH_{\mq}(\mX_p,\mU)}_{4'}\big) \\
 &\times \bigg(\sum_{i=1}^n\norm[1]{1 + \norm[0]{\grad f_i(\mX_p)}^2}_2^{1/2} \bigg.\\
 &\quad\quad+ \bigg.\sum_{i=1}^m\norm[2]{1 + \norm[1]{\mJ_\mf^{\mId}(\mX_p)(\grad h_i)[\mf(\mX_p) + \mU]}^2}_2^{1/2} \bigg), \\
    C_2 ={}& \Big[ \norm[0]{\expec(\mY) - \mm_{\widetilde{\mY}}}^2 + \trace \mA_2 + \trace\mSigma_\mV + \trace\mP_\mX + \trace\mP_{\widetilde{\mY}} \\
  &\quad\quad- 2\trace\big(\sqrt{\smash[b]{\mL (\mA_1 + \mB_1 + \mC_\mV)\mL}} \, \big) \Big]^{1/2} .
  \end{split}
\end{equation*}
Here $\mA_1,\mB_1,\mC_\mV$ and $\mA_2$ are defined in \cref{lemma:CovLemma}, $\mJ_\mf^\mId(\mX_p)$ is defined in \Cref{eq:gradqi}, and \smash{$\mL = \cov(\widetilde{\mZ})^{1/2}$}.
\end{theorem}
\begin{proof}
  The constant $C_1$ is obtained from \cref{lemma:CovLemma,lemma:PoincareLemma}.
  The constant $C_2$ is obtained by estimating $\was(\pr\mZ,\widetilde{\mZ})$ through~\eqref{eq:WasGBound} and \cref{lemma:CovLemma} and noting that $\trace \cov(\mZ) \leq \trace(\mA_2 + \mC_\mV)$,  
\begin{align*}
\expec(\mZ) - \expec(\widetilde{\mZ}) = \begin{pmatrix} \mZero_{n\times 1} \\ \expec(\mY) - \mm_{\widetilde{\mY}}\end{pmatrix}, 
\end{align*}
and
\begin{equation*}
\trace \cov(\widetilde{\mZ})  = \trace\begin{pmatrix} \mP_\mX & \mC_{\widetilde{\mX},\widetilde{\mY}} \\ \mC_{\widetilde{\mX},\widetilde{\mY}}^\transpose & \mP_{\widetilde{\mY}}\end{pmatrix} = \trace\big(\mP_\mX + \mP_{\widetilde{\mY}}\big). \qedhere
\end{equation*}
\end{proof}

Let us examine the upper bound of \Cref{thm:mainTheorem} when $\mf$ and $\mh$ are affine.
In this case the function $\mq$ is also affine, so that $C_1 = 0$ because the Hessians are zero.
We are left with
\begin{equation*} 
  \begin{split}
    \was({}&\mZ,\widetilde{\mZ}) \leq C_2 .
  \end{split}
\end{equation*}
Note that $C_1$ is positive if only either $\mf$ or $\mh$ is affine.
Since the model and measurement functions are affine, $\mX = \widetilde{\mX}$ and $\mY = \widetilde{\mY}$ are Gaussian random variables.
The bounds in \Cref{lemma:CovLemma} coincide because Jacobians are constant matrices.
Hence $\cov(\mZ) = \cov(\widetilde{\mZ}) = \mA_1 + \mB_1 + \mC_{\mV} = \mA_2 + \mC_{\mV}$ and $\trace \cov(\mZ)  = \trace \mP_\mX + \trace \mP_\mY$.
Also recall from~\eqref{eq:mC-def} that $\trace \mC_\mV = \trace \mSigma_\mV$.
These equations yield (note the square)
\begin{equation*}
  \begin{split}
    \was(\mZ,{}&\widetilde{\mZ})^{2} \\
    \leq{}& \norm[0]{\expec(\mY) - \expec(\widetilde{\mY})}^2 + \trace \mA_2 + \trace\mSigma_\mV + \trace\mP_\mX + \trace\mP_{\widetilde{\mY}} \\
    &- 2\trace\big(\sqrt{\smash[b]{\mL (\mA_1 + \mB_1 + \mC_\mV)\mL}} \, \big) \\
    ={}& \norm[0]{\expec(\mY) - \expec(\mY)}^2 + \trace \cov(\mZ) + \trace \cov(\mZ) \\
    &- 2\trace\big(\sqrt{\smash[b]{\cov(\mZ)^{1/2} \cov(\mZ) \cov(\mZ)^{1/2}}} \, \big),
  \end{split}
\end{equation*}
which is zero since $\sqrt{\mA^{1/2} \mA \mA^{1/2}} = \mA$.

\section{A numerical example} \label{sec:numerical-example}

\begin{figure}
  \centering
  \includegraphics[width=\columnwidth]{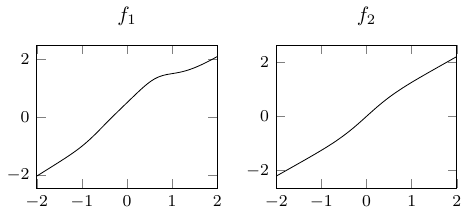}
  \caption{The non-linear functions $f_1$ and $f_2$ in~\eqref{eq:ex-funcs}.}
  \label{Fig:fs}
  \end{figure}

We consider two variants of the univariate non-stationary growth model (e.g.,~\cite{Gordon1993}) commonly used to assess and compare non-linear Kalman filters.
The models we consider are
\begin{equation} \label{eq:UNGM}
    \begin{split}
      X = f_i(X_p) + U, \quad Y = \tfrac{1}{2} X^2 + V
    \end{split}
\end{equation}
for the dynamic model functions (plotted in \Cref{Fig:fs})
\begin{equation} \label{eq:ex-funcs}
  f_1(x) = x + \frac{1+x}{2(1+x^4)} \: \text{ and } \: f_2(x) = x + \frac{x}{2(1+x^2)}.
\end{equation}
The function $f_1$ is, at least visually, more non-linear than $f_2$.
We set $U, V \sim \gauss(0, 0.05)$ and $X_p \sim \gauss(0.2, 1)$.
\Cref{Fig:Example-2} shows the distributions of $X$, $Y$, $\widetilde{X}$, and $\widetilde{Y}$.
With the expectations computed using Monte Carlo with $10^6$ samples, \Cref{thm:mainTheorem} gives the bounds $\was(\mZ, \widetilde{\mZ}) \leq 2\,267$ and $\was(\mZ, \widetilde{\mZ}) \leq 663$ for $f_1$ and $f_2$, respectively.
As expected, the bound is larger for the more non-linear function (i.e., $f_1$).
We note that the term $C_1$ of \Cref{thm:mainTheorem} dominates as $C_2 \approx \sqrt{2}$ for both functions.
The bounds are very conservative: the true distances are $\was(X, \widetilde{X}) \approx 0.21$ and $\was(Y, \widetilde{Y}) \approx 0.45$ for $f_1$, and $\was(X, \widetilde{X}) \approx 0.09$ and $\was(Y, \widetilde{Y}) \approx 0.48$ for $f_2$.
We have not found a convenient way to compute $\was(\mZ, \widetilde{\mZ})$ from samples.

\begin{figure}
\centering
\includegraphics[width=0.48\columnwidth]{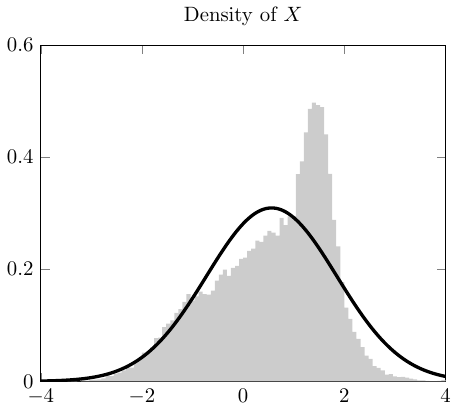}
\includegraphics[width=0.48\columnwidth]{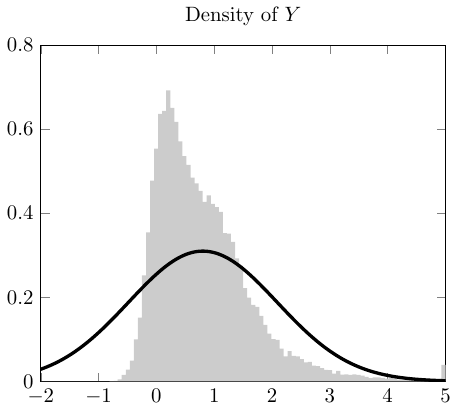}
\end{figure}

\begin{figure}
\centering
\includegraphics[width=0.48\columnwidth]{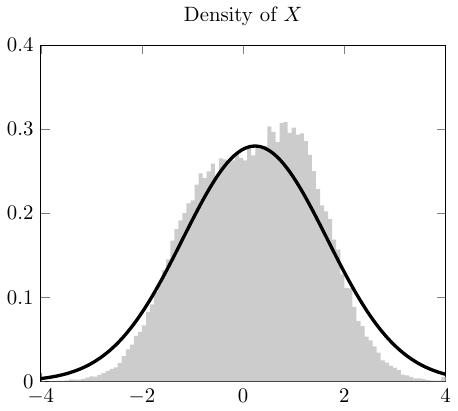}
\includegraphics[width=0.48\columnwidth]{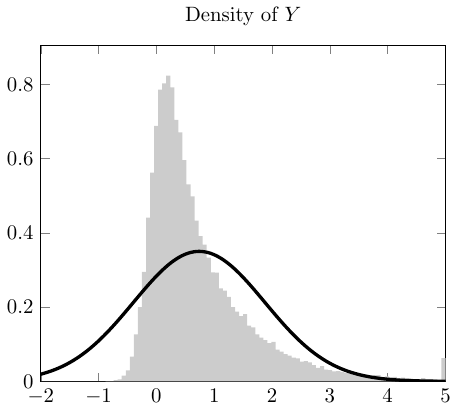}
\caption{Distributions of $X$ and $Y$ for the model in~\eqref{eq:UNGM} with $f_1$ (above) and $f_2$ (below). Densities of the Gaussian approximations \smash{$\widetilde{X}$} and \smash{$\widetilde{Y}$} are in black.}
\label{Fig:Example-2}
\end{figure}

\section{Conclusion}

\Cref{thm:mainTheorem} gives an upper bound on the Wasserstein distance between the true joint distribution $\mZ$ of the prediction $\mX$ and measurement $\mY$ and its Gaussian approximation $\widetilde{\mZ} = (\widetilde{\mX}, \widetilde{\mY})$.
The bound vanishes when the model and measurement functions are affine.
Understanding how the conditional filtering distribution $\mX \mid \mY = \my$ relates to its Gaussian approximation $\widetilde{\mX} \mid \widetilde{\mY} = \my$ is much more challenging.
The key difficulty is in quantifying how perturbing the prior affects the posterior~\cite[Lem.\@~3.6]{Kunsch2001}.
Unfortunately, such stability bounds are available only under very restrictive assumptions (e.g.,\@~\cite{LeyReinertSwan2015} and \cite{Helin2023}).
This is a pertinent problem throughout Bayesian statistics.


\end{document}